\documentclass[twoside]{article}

\usepackage{arxiv}

\usepackage[utf8]{inputenc}
\usepackage[T1]{fontenc}
\usepackage{url}
\usepackage{booktabs}
\usepackage{natbib}
\usepackage{amsmath}
\usepackage{amsfonts}
\usepackage{amsthm}
\usepackage{microtype}
\usepackage{graphicx}
\usepackage{capt-of}
\usepackage{algorithm}
\usepackage{algpseudocode}
\usepackage{amsopn}
\usepackage{hyperref}
\usepackage{cleveref}

\DeclareMathOperator{\rank}{rank}

\DeclareMathOperator{\Sym}{Sym}
\DeclareMathOperator{\tr}{tr}

\newtheorem{theorem}{Theorem}
\newtheorem{proposition}[theorem]{Proposition}
\newtheorem{corollary}[theorem]{Corollary}

\theoremstyle{definition}
\newtheorem{definition}[theorem]{Definition}
\newtheorem{remark}[theorem]{Remark}

\crefname{theorem}{theorem}{theorems}
\crefname{proposition}{proposition}{propositions}
\crefname{corollary}{corollary}{corollaries}
\crefname{definition}{definition}{definitions}
\crefname{remark}{remark}{remarks}
\crefname{algorithm}{algorithm}{algorithms}

\hypersetup{hidelinks}

\newcommand{\Hnd}{\mathcal{H}_{n,d}}
\newcommand{\Fn}{\mathcal{F}_{n,d,r}}
\newcommand{\Vn}{\mathcal{V}_{n,d,r}}

\shorttitle{Algebraic Expressivity Certificates}
\shortauthor{Akbari,~~Jamshidi}

\title{Algebraic Expressivity Certificates for Shallow Polynomial Neural Networks}

\author{
 Sepehr Akbari \\
  Dept. of Math \& Computer Science\\
  Lake Forest College\\
  \texttt{akbaris79@lakeforest.edu}
  \And
  Shahrzad Jamshidi \\
  Dept. of Math \& Computer Science\\
  Lake Forest College\\
  \texttt{sjamshidi@lakeforest.edu}
}

\begin{document}
\maketitle

\begin{abstract}
We study exact representability by bias-free shallow polynomial neural networks using algebraic geometry. Over $\mathbb{C}$, a width-$r$ network with activation $z\mapsto z^d$ computes a sum of $r$ $d$-th powers of linear forms, whose Zariski closure is a Veronese secant variety. Ideal elimination therefore yields polynomial certificates of nonrepresentability. We implement this construction as a generic architecture-to-certificate pipeline. For quadratics, we recover the exact symmetric determinantal description and explain its dimension through orthogonal symmetry. In higher degree, the implementation recovers classical catalecticant and secant equations and maps the practical reach of direct elimination across a finite architecture sweep. We also derive the exact population loss floor for a rank-two quadratic network on the sphere, illustrating how an algebraic obstruction induces irreducible approximation error.
\end{abstract}

\keywords{polynomial neural networks, algebraic geometry, expressivity certificates, Veronese secant varieties, Waring rank}

\section{Introduction}
\label{sec:introduction}

Universal approximation does not describe the exact function class of a fixed finite network. Moreover, the standard shallow universal-approximation theorem requires a non-polynomial activation, together with its usual regularity and bias assumptions \citep{Leshno1993}. A fixed activation $z^d$ therefore does not become universal merely by increasing width. In the bias-free model, every output remains a homogeneous polynomial of degree $d$. The question here is which degree-$d$ forms are represented by a prescribed width?

This distinction matters even when exact polynomial representation is not itself the learning objective. An observed residual may reflect optimization, finite data, or an architectural obstruction. Ordinary asymptotic approximation statements do not isolate the last possibility at a fixed width. Polynomial parameterizations make that obstruction accessible to exact equations. A violated invariant can prove that a target lies outside the architecture, independently of how the weights were optimized. This does not reduce a general learning problem to polynomial membership; it supplies a precise diagnostic for the polynomial model.

For a monomial activation, the weights map polynomially to output coefficients, identifying shallow networks with symmetric tensor decompositions. Our foundational reference is classical tools of algebraic geometry, including Waring decompositions, Veronese secant varieties, Alexander--Hirschowitz dimension theory, symmetric determinantal ideals, and catalecticants \citep{IarrobinoKanev1999,Landsberg2012,BrambillaOttaviani2008,Jozefiak1978}. Within the literature, polynomial networks have also been studied through algebraic geometry \citep{Kileel2019}, along with the real geometry of neuromanifolds and their Zariski closures \citep{Kubjas2024}, and distribution-dependent optimization of shallow polynomial networks \citep{Arjevani2026}.

Against this background, we first make the image--closure distinction explicit for shallow polynomial networks, since it determines exactly what polynomial equations can certify. Elimination of the graph ideal~\footnote{We use the term graph-ideal elimination since we perform elimination on the ideal generated from the constructed graph in \cref{sec:algebraic}, rather than an arbitrary ideal in the coefficient space.} computes equations of the border-rank locus; a violated equation proves nonrepresentability, whereas simultaneous vanishing generally does not decide exact Waring rank. In the quadratic case, where bounded rank is closed, we give a self-contained determinantal characterization and connect its dimension formula to the symmetry $V\mapsto QV$ with $Q\in O(r,\mathbb{C})$. This makes the $\binom r2$-dimensional parameter redundancy visible directly at the network level. The result is a precise certificate interpretation anchored in the geometry of symmetric tensors and Veronese secant varieties \citep{IarrobinoKanev1999,Landsberg2012,Jozefiak1978}.

We then implement graph-ideal elimination as a general pipeline for any specified $(n,d,r)$, together with a target evaluator that turns a violated equation into a certificate. We summarize the recorded computations in a machine-readable atlas and compare every completed case with the quadratic dimension formula or the Alexander--Hirschowitz theorem \citep{BrambillaOttaviani2008}. The pipeline recovers structured classical ideals in representative cases, including the binary-quartic catalecticant determinant and the maximal-minor ideal of the $(3,3,2)$ ternary-cubic architecture; the latter equality is checked exactly over $\mathbb{Q}$. Finally, we specialize the distribution-induced metric studied for shallow polynomial networks \citep{Arjevani2026} to the recorded sphere experiment and derive its exact population loss floor. These pieces connect the algebraic description of expressivity to an executable certificate procedure and to a concrete approximation barrier.

\section{Algebraic formulation}
\label{sec:algebraic}

We begin by translating the network into a polynomial map on the coefficient space. This separates exact representability from the algebraic closure on which implicit equations naturally live. Let
\[
    \Hnd:=\mathbb{C}[x_1,\ldots,x_n]_d
    \simeq \Sym^d((\mathbb{C}^n)^*)
\]
be the $N=\binom{n+d-1}{d}$ dimensional space of homogeneous degree-$d$ forms. A width-$r$ network with scalar output weights computes
\begin{equation}
    f_{a,W}(x)=\sum_{i=1}^r a_i(w_i^\top x)^d.
    \label{eq:signed-network}
\end{equation}
Over $\mathbb{C}$, choose a $d$-th root of each $a_i$ and absorb it into $w_i$. The complex function class is therefore parameterized by
\begin{equation}
    \phi_{n,d,r}:(\mathbb{C}^n)^r\longrightarrow \Hnd,
    \qquad
    (v_1,\ldots,v_r)\longmapsto \sum_{i=1}^r(v_i^\top x)^d.
    \label{eq:parameterization}
\end{equation}
Here $v_i^\top x=\sum_jv_{ij}x_j$, without any complex conjugation used. If $\alpha\in\mathbb{N}^n$ satisfies $|\alpha|=d$ and $P=\sum_{|\alpha|=d}c_\alpha x^\alpha$, then
\begin{equation}
    c_\alpha=\binom{d}{\alpha}\sum_{i=1}^r v_i^\alpha.
    \label{eq:coefficient-map}
\end{equation}

\begin{definition}[Image and closure]
The actual complex network class and its Zariski closure are
\begin{equation}
    \Fn:=\operatorname{Im}(\phi_{n,d,r}),
    \qquad
    \Vn:=\overline{\Fn}^{\,Z}\subseteq\Hnd.
    \label{eq:image-closure}
\end{equation}
Thus $\Fn$ is the locus of complex symmetric (Waring) \textit{rank} at most $r$, whereas $\Vn$ is the locus of complex symmetric \textit{border rank} at most $r$.
\end{definition}

\begin{proposition}[Secant-variety interpretation]
\label{prop:secant}
Let $\nu_d:\mathbb{P}^{n-1}\to\mathbb{P}^{N-1}$ be the degree-$d$ Veronese embedding. Then $\Vn$ is the affine cone over $\sigma_r(\nu_d(\mathbb{P}^{n-1}))$, the $r$-th secant variety of the Veronese variety.
\end{proposition}

\begin{proof}
The cone over the Veronese variety consists of pure powers of linear forms. Sums of $r$ such powers give $\Fn$, and the secant variety is the Zariski closure of their projectivization \citep{IarrobinoKanev1999,Landsberg2012}.
\end{proof}

The closure in \cref{eq:image-closure} is important. For $d\geq3$, bounded-rank symmetric tensors do not need to form a closed set \citep{Comon2008,Arjevani2026}. Consequently, equations of $\Vn$ certify nonmembership in $\Fn$, but their simultaneous vanishing does not generally produce an exact width-$r$ decomposition.

To pass from the parameterization to equations in coefficient space, we eliminate the weights from the graph of $\phi_{n,d,r}$. Write the coordinates of \cref{eq:parameterization} as $\phi_1(v),\ldots,\phi_N(v)$ and introduce coefficient variables $c=(c_1,\ldots,c_N)$. We can then define the graph ideal
\begin{equation}
    J=\langle c_j-\phi_j(v):1\leq j\leq N\rangle
    \subset\mathbb{C}[v,c].
    \label{eq:graph-ideal}
\end{equation}

\begin{proposition}[Graph-ideal elimination]
\label{prop:elimination}
For any polynomial map $\phi:\mathbb{C}^m\to\mathbb{C}^N$ and the graph ideal in \cref{eq:graph-ideal},
\[
    J\cap\mathbb{C}[c]
    =I\!\left(\overline{\operatorname{Im}\phi}^{\,Z}\right).
\]
In particular, a Gr\"obner basis for $J$ in an elimination order yields generators for $I(\Vn)$.
\end{proposition}

\begin{proof}
The substitution homomorphism $\psi:\mathbb{C}[c]\to\mathbb{C}[v]$, $c_j\mapsto\phi_j(v)$, has kernel $J\cap\mathbb{C}[c]$. A polynomial lies in this kernel exactly when it vanishes on $\operatorname{Im}\phi$, hence exactly when it vanishes on its Zariski closure. The computational statement is the elimination theorem \citep{Cox2015}.
\end{proof}

\begin{corollary}[Polynomial nonmembership certificate]
\label{cor:certificate}
If $g\in I(\Vn)$ and $g(P)\neq0$, then
\[
    P\notin\Vn\quad\Longrightarrow\quad P\notin\Fn.
\]
The converse is not necessarily true.
\end{corollary}

The generic code realizes this construction directly. It enumerates weak compositions $\alpha$ of $d$, creates the $N$ coefficient variables, constructs the equations in \cref{eq:coefficient-map} for arbitrary supplied $(n,r,d)$, forms $J$ over $\mathbb{Q}$, and eliminates all $rn$ weight variables. The fixed quadratic scripts therefore serve as checks of the same general construction.

Before computing equations, dimension already determines whether a nonzero vanishing ideal can exist. The Alexander--Hirschowitz theorem gives the generic dimension of these closures. With $n$ denoting the number of variables, for $d\geq3$ one has
\[
    \dim\Vn=\min\{rn,N\}
\]
except for
\[
    (n,d,r)\in\{(3,4,5),(4,4,9),(5,3,7),(5,4,14)\},
\]
where the affine dimension is one below the expected value \citep{BrambillaOttaviani2008}. If $\dim\Vn=N$, then $\Vn=\Hnd$ and $I(\Vn)=0$. This does not imply $\Fn=\Hnd$. A full-dimensional closure can contain limits whose exact rank exceeds $r$.

\begin{remark}[Real architectures]
\label{rem:real}
Absorbing $a_i$ in \cref{eq:signed-network} is valid over $\mathbb{C}$. Over $\mathbb{R}$ it fails for even $d$ and negative $a_i$. The signed real architecture \cref{eq:signed-network}, the absorbed architecture $\sum_i(v_i^\top x)^d$, and the positive-semidefinite quadratic model in \cref{sec:optimization} are therefore distinct. A real target outside $\Vn$ is outside every corresponding real width-$r$ class, but membership in $\Vn(\mathbb{R})$ doesn't need to imply representability by a chosen real architecture. Complete real descriptions can require polynomial inequalities as well as equalities \citep{Kubjas2024}.
\end{remark}

\section{Quadratic networks}
\label{sec:quadratic}

The quadratic case is exceptional because bounded symmetric rank is already closed, so the algebraic certificates characterize exact complex representability. For $d=2$, identify $P(x)=x^\top Mx$ with $M\in\Sym_n(\mathbb{C})$. Under the monomial coefficient convention, $M_{ii}=c_{2e_i}$ and $M_{ij}=c_{e_i+e_j}/2$ for $i\neq j$.

\begin{theorem}[Quadratic determinantal characterization]
\label{thm:quadratic}
Let $s=\min(r,n)$. Then
\begin{equation}
    \mathcal{F}_{n,2,r}=\mathcal{V}_{n,2,r}
    =\{M\in\Sym_n(\mathbb{C}):\rank(M)\leq s\}.
    \label{eq:quadratic-variety}
\end{equation}
If $s<n$, its defining ideal is generated by the $(s+1)\times(s+1)$ minors of the generic symmetric matrix. If $s=n$, the ideal is zero.
\end{theorem}

\begin{proof}
For $V\in\mathbb{C}^{r\times n}$, the network output is
\[
    \sum_{i=1}^r(v_i^\top x)^2=x^\top V^\top Vx,
\]
so its matrix has rank at most $s$. Conversely, let $M$ be complex symmetric of rank $q\leq s$. The Autonne--Takagi factorization gives $M=U\Sigma U^\top$ \citep{HornZhang2012}. Removing the zero singular values and setting $B=\Sigma_q^{1/2}U_q^\top\in\mathbb{C}^{q\times n}$ gives $M=B^\top B$; pad $B$ with zero rows if $q<r$. Hence the image is the rank-$\leq s$ locus, which is closed. Its ideal is the classical symmetric determinantal ideal \citep{Jozefiak1978}.
\end{proof}

\begin{corollary}[Dimension]
\label{cor:quadratic-dimension}
For $s=\min(r,n)$,
\begin{equation}
    \dim\mathcal{V}_{n,2,r}=ns-\binom{s}{2}.
    \label{eq:quadratic-dimension}
\end{equation}
In particular, for $r\leq n$ the dimension is $nr-\binom{r}{2}$.
\end{corollary}

\begin{proof}
The rank-exactly-$s$ stratum is dense in \cref{eq:quadratic-variety}. Choose its $(n-s)$-dimensional kernel, contributing $s(n-s)$ parameters, and a nondegenerate symmetric form on the resulting $s$-dimensional quotient, contributing $s(s+1)/2$. Their sum is \cref{eq:quadratic-dimension}.
\end{proof}

\begin{remark}[Orthogonal redundancy]
\label{rem:orthogonal-symmetry}
The dimension loss has a direct network-level mechanism. For
\[
    O(r,\mathbb{C})=\{Q\in\mathbb{C}^{r\times r}:Q^\top Q=I_r\},
\]
the transformation $V\mapsto QV$ leaves $V^\top V$ unchanged. Hence every fiber contains an orthogonal-group orbit. When $r\leq n$ and $V$ has full row rank, the stabilizer of $V$ is trivial, so this orbit has dimension
\[
    \dim O(r,\mathbb{C})=\binom r2.
\]
The parameter space has dimension $rn$, and Corollary~\ref{cor:quadratic-dimension} shows that the generic fiber loses exactly these $\binom r2$ orthogonal directions; $rn-\binom r2$ dimensions remain in the image.
\end{remark}

The smallest nontrivial quadratic example makes the determinantal certificate explicit. For $(n,r)=(3,2)$, write
\begin{equation}
    M(c)=
    \begin{pmatrix}
        c_{200} & c_{110}/2 & c_{101}/2\\
        c_{110}/2 & c_{020} & c_{011}/2\\
        c_{101}/2 & c_{011}/2 & c_{002}
    \end{pmatrix}.
    \label{eq:ternary-matrix}
\end{equation}

\begin{corollary}[Explicit determinant]
\label{cor:ternary-determinant}
A ternary quadratic belongs to $\mathcal{F}_{3,2,2}$ if and only if
\begin{align}
    4\det M(c)
    ={}&4c_{200}c_{020}c_{002}-c_{200}c_{011}^2-c_{020}c_{101}^2\notag\\
       &{}-c_{002}c_{110}^2+c_{110}c_{101}c_{011}=0.
    \label{eq:determinant-certificate}
\end{align}
\end{corollary}

\begin{proof}
Apply \cref{thm:quadratic} with $s=2$. The only $3\times3$ minor is the determinant, whose expansion is \cref{eq:determinant-certificate}.
\end{proof}

The generic elimination code recovers \cref{eq:determinant-certificate} as one returned generator. The target $P(x)=x_1^2+x_2^2+x_3^2$ has matrix $I_3$ and violates the equation, so it is not exactly representable by any complex width-two quadratic network. For signed real output weights, the same rank characterization follows from the spectral theorem. The absorbed real model used below is more restrictive since it contains only positive-semidefinite matrices $V^\top V$ of rank at most two.

\section{Higher-degree certificates}
\label{sec:higher}

Let $W=(\mathbb{C}^n)^*$ and $P\in\Sym^d W$. Contraction by a degree-$k$ differential form defines the catalecticant map
\begin{equation}
    \operatorname{Cat}_{k,d-k}(P):\Sym^k(W^*)\longrightarrow\Sym^{d-k}W,
    \qquad D\longmapsto D\mathbin{\lrcorner}P.
    \label{eq:catalecticant-map}
\end{equation}
If $P=\ell^d$, the image of this map is contained in the line spanned by $\ell^{d-k}$, so the catalecticant has rank at most one. Therefore
\[
    P=\sum_{i=1}^r\ell_i^d
    \quad\Longrightarrow\quad
    \rank\operatorname{Cat}_{k,d-k}(P)\leq r.
\]
All $(r+1)\times(r+1)$ minors consequently vanish on $\mathcal{F}_{n,d,r}$ and, being polynomial, on $\mathcal{V}_{n,d,r}$. Thus catalecticant minors are systematic closure certificates. They do not define every Veronese secant variety in every degree, but in the cases below classical theory proves that they generate the full ideal \citep{IarrobinoKanev1999,LandsbergOttaviani2013,Raicu2013}.

\paragraph{Certificate atlas.}

The generic implementation constructs \cref{eq:graph-ideal} over $\mathbb{Q}$, uses lexicographic order, and invokes Macaulay2's \texttt{eliminate} function \citep{Macaulay2}. Polynomial-ideal problems have doubly exponential worst-case bounds \citep{MayrMeyer1982}, but such bounds do not predict a particular structured instance. Accordingly, we treat algebraic correctness and observed completion time as separate questions.

The primary recorded log contains all 100 tuples with $n,r\in\{2,3,4,5,6\}$ and $d\in\{2,3,4,5\}$ under a 3600-second per-instance cutoff. A parser extracts one row per architecture and computes the corresponding theoretical dimension and codimension. Among the 33 completed runs, 21 returned the zero ideal and 12 returned nonzero ideals. In every completed case, zero versus nonzero agrees with the theoretical prediction: zero occurs exactly when the quadratic formula or Alexander--Hirschowitz theorem predicts $\dim\mathcal{V}_{n,d,r}=N$. The original text log did not record ideal dimensions directly, so zero versus nonzero is the strongest invariant that can be compared uniformly across all completed runs.

\begin{center}
\centering
\captionof{table}{Theory-versus-computation summary for the 3600-second sweep. ``Zero'' and ``nonzero'' refer to the elimination ideal returned by Macaulay2.}
\label{tab:benchmark-summary}
\begin{tabular}{@{}ccccc@{}}
\toprule
Degree $d$ & Completed & Zero & Nonzero & Timed out \\
\midrule
2 & 20 & 12 & 8 & 5 \\
3 & 6  & 5  & 1 & 19 \\
4 & 5  & 3  & 2 & 20 \\
5 & 2  & 1  & 1 & 23 \\
\midrule
Total & 33 & 21 & 12 & 67 \\
\bottomrule
\end{tabular}
\end{center}

\Cref{tab:selected-atlas} reports representative completed cases. Here $K=rn+N$ is the number of variables in the graph-ideal ring before elimination. ``Returned generators'' describes Macaulay2's output and does not assert minimality. The full machine-readable atlas includes all 100 cases and the distinct degrees represented in every successful output.

\begin{center}
\centering
\captionof{table}{Selected entries from the certificate atlas. The recorded runtimes are not portable benchmarks.}
\label{tab:selected-atlas}
\small
\begin{tabular}{@{}cccccccc@{}}
\toprule
$(n,d,r)$ & $N$ & $rn$ & $K$ & Pred. codim. & Returned gens. & Degrees & Time (s) \\
\midrule
$(3,2,2)$ & 6  & 6  & 12 & 1 & 1   & 3 & 0.50 \\
$(4,2,2)$ & 10 & 8  & 18 & 3 & 10  & 3 & 0.52 \\
$(4,2,3)$ & 10 & 12 & 22 & 1 & 1   & 4 & 0.74 \\
$(6,2,3)$ & 21 & 18 & 39 & 6 & 105 & 4 & 1303.97 \\
$(3,3,2)$ & 10 & 6  & 16 & 4 & 20  & 3 & 1.21 \\
$(2,4,2)$ & 5  & 4  & 9  & 1 & 1   & 3 & 0.49 \\
$(3,4,2)$ & 15 & 6  & 21 & 9 & 148 & 3 & 38.21 \\
$(2,5,2)$ & 6  & 4  & 10 & 2 & 4   & 3 & 0.50 \\
\bottomrule
\end{tabular}
\end{center}

For $(3,4,2)$, the direct computation already returns 148 cubics at $K=21$, illustrating how quickly an explicit generating set can become unwieldy. By contrast, the $(3,3,2)$ output below has a compact determinantal description. The 67 timeouts delimit only this direct implementation under the stated cutoff; they are not impossibility results for Gr\"obner methods or specialized secant-variety algorithms.

\paragraph{Binary quartics.}

Let
\[
P(x_1,x_2)=c_{40}x_1^4+c_{31}x_1^3x_2+c_{22}x_1^2x_2^2+c_{13}x_1x_2^3+c_{04}x_2^4.
\]
In divided-power coordinates its middle catalecticant is
\begin{equation}
    \operatorname{Cat}_{2,2}(P)=
    \begin{pmatrix}
        c_{40} & c_{31}/4 & c_{22}/6\\
        c_{31}/4 & c_{22}/6 & c_{13}/4\\
        c_{22}/6 & c_{13}/4 & c_{04}
    \end{pmatrix}.
    \label{eq:quartic-catalecticant}
\end{equation}

\begin{proposition}[Binary-quartic closure certificate]
\label{prop:quartic}
The Zariski closure $\mathcal{V}_{2,4,2}$ is the hypersurface defined by $\det\operatorname{Cat}_{2,2}(P)=0$. After clearing denominators, its equation is
\begin{equation}
    72c_{40}c_{22}c_{04}-27c_{40}c_{13}^2-27c_{31}^2c_{04}
    +9c_{31}c_{22}c_{13}-2c_{22}^3=0.
    \label{eq:quartic-certificate}
\end{equation}
The generic pipeline returns this cubic equation.
\end{proposition}

\begin{proof}
The rank argument following \cref{eq:catalecticant-map} gives the necessary determinant equation. Classical results for the rational normal quartic show that it defines the second secant variety \citep{IarrobinoKanev1999,LandsbergOttaviani2013}. Expanding \cref{eq:quartic-catalecticant} and clearing denominators gives \cref{eq:quartic-certificate}.
\end{proof}

\begin{remark}[Border rank is not exact rank]
The vanishing of \cref{eq:quartic-certificate} characterizes membership in $\mathcal{V}_{2,4,2}$, not in $\mathcal{F}_{2,4,2}$. For example,
\[
    \frac{(x_1+\varepsilon x_2)^4-x_1^4}{4\varepsilon}
    \longrightarrow x_1^3x_2
    \qquad(\varepsilon\to0),
\]
so $x_1^3x_2$ has border rank at most two, while its exact Waring rank is four \citep{Comon2008,Arjevani2026}. It satisfies the determinant equation but is not exactly representable at width two.
\end{remark}

\paragraph{Ternary cubics.}

For a ternary cubic $P=\sum_{i+j+k=3}c_{ijk}x_1^ix_2^jx_3^k$, transpose the map $\operatorname{Cat}_{1,2}(P)$ and write it in divided-power coordinates as
\begin{equation}
    C(P)=
    \begin{pmatrix}
    c_{300} & c_{210}/3 & c_{201}/3 & c_{120}/3 & c_{111}/6 & c_{102}/3\\
    c_{210}/3 & c_{120}/3 & c_{111}/6 & c_{030} & c_{021}/3 & c_{012}/3\\
    c_{201}/3 & c_{111}/6 & c_{102}/3 & c_{021}/3 & c_{012}/3 & c_{003}
    \end{pmatrix}.
    \label{eq:ternary-cubic-catalecticant}
\end{equation}

\begin{proposition}[Ternary-cubic secant equations]
\label{prop:ternary-cubic}
The ideal $I(\mathcal{V}_{3,3,2})$ is generated by the 20 maximal minors of $C(P)$. The 20 cubic generators returned by the generic elimination pipeline generate the same ideal.
\end{proposition}

\begin{proof}
The minors vanish by the catalecticant rank argument. The classical theorem on the second secant of a Veronese variety identifies these $3\times3$ minors as generators of the full ideal \citep{Raicu2013}. For the computational comparison, denominators in \cref{eq:ternary-cubic-catalecticant} are retained and all 20 minors are expanded over $\mathbb{Q}$. The logged cubics span a 20-dimensional coefficient subspace, the minors span a 20-dimensional subspace, and their union also has rank 20. Hence the two cubic spans, and therefore the ideals they generate, are equal.
\end{proof}

Thus, for $(3,3,2)$, the general construction recovers a nontrivial higher-degree secant ideal in a recognizable determinantal form. Since $N=10$ and the predicted affine dimension is $rn=6$, its codimension is four, in agreement with the known geometry.

\paragraph{Target certificates.}

The companion evaluator accepts $(n,d,r)$, a completed benchmark log, and an exact rational target coefficient vector. It evaluates every returned generator and reports a violated invariant when one exists. By Corollary~\ref{cor:certificate}, any reported violation has the implication
\[
    g(P)\neq0
    \quad\Longrightarrow\quad
    P\notin\mathcal{V}_{n,d,r}
    \quad\Longrightarrow\quad
    P\notin\mathcal{F}_{n,d,r}.
\]
If all returned equations vanish, the utility reports membership in the computed closure and leaves exact rank undecided for $d\geq3$. It makes the stronger exact-membership conclusion only for $d=2$, where \cref{thm:quadratic} proves image equals closure. Tests certify $x_1^2+x_2^2+x_3^2$ outside $\mathcal{V}_{3,2,2}$ and $x_1^3+x_2^3+x_3^3$ outside $\mathcal{V}_{3,3,2}$; a two-term ternary cubic exercises the deliberately inconclusive higher-degree branch.

\section{Sphere loss floor}
\label{sec:optimization}

We finally connect the algebraic obstruction to a learning objective. The relevant metric is induced by the input distribution rather than by the ambient coefficient coordinates. Coefficient-space Frobenius distance is not generally the population mean-squared error. A distribution $\mu$ on inputs give
\begin{equation}
    \langle P,Q\rangle_\mu=\mathbb{E}_{x\sim\mu}[P(x)Q(x)],
    \qquad
    L_\mu(f,P)=\|f-P\|_\mu^2.
    \label{eq:distribution-metric}
\end{equation}
This is a low-rank tensor approximation problem in a metric determined by moments of $\mu$ \citep{Arjevani2026}. The calculation below is a concrete specialization of that general viewpoint.

The recorded experiment uses $x\sim\operatorname{Unif}(S^2)$, target $P(x)=x_1^2+x_2^2+x_3^2=x^\top I_3x$, and the absorbed real network
\[
    f_V(x)=\sum_{i=1}^2(v_i^\top x)^2=x^\top Mx,
    \qquad M=V^\top V\succeq0,\qquad\rank(M)\leq2.
\]

\begin{proposition}[Exact sphere loss floor]
\label{prop:sphere}
For a symmetric $3\times3$ matrix $A$ and $x\sim\operatorname{Unif}(S^2)$,
\begin{equation}
    \mathbb{E}\big[(x^\top Ax)^2\big]
    =\frac{2\tr(A^2)+\tr(A)^2}{15}.
    \label{eq:sphere-moment}
\end{equation}
Among positive-semidefinite matrices $M$ of rank at most two, the minimum of
\[
    \mathbb{E}\big[(x^\top(M-I_3)x)^2\big]
\]
is attained precisely by matrices whose eigenvalues are $(5/4,5/4,0)$, and the minimum is
\[
    L_\star=\frac16.
\]
\end{proposition}

\begin{proof}
Rotational invariance gives
\[
\mathbb{E}[x_ix_jx_kx_l]
=\frac{\delta_{ij}\delta_{kl}+\delta_{ik}\delta_{jl}+\delta_{il}\delta_{jk}}{3(3+2)},
\]
which yields \cref{eq:sphere-moment} by contraction. Since the objective is orthogonally invariant, write the eigenvalues of $M$ as $(\lambda_1,\lambda_2,0)$ with $\lambda_1,\lambda_2\geq0$. Multiplying the loss by $15$ gives
\[
2\big((\lambda_1-1)^2+(\lambda_2-1)^2+1\big)
+(\lambda_1+\lambda_2-3)^2.
\]
This strictly convex quadratic has stationary equations $3\lambda_1+\lambda_2=5$ and $\lambda_1+3\lambda_2=5$, hence $\lambda_1=\lambda_2=5/4$. The point is feasible and gives $15L_\star=5/2$.
\end{proof}

The numerical script samples $10^4$ points by normalizing standard Gaussian vectors, fixes seed $49$, and applies full-batch Stochastic Gradient Descent (SGD) for $60{,}000$ iterations with learning rate $10^{-4}$. The final empirical MSE in \cref{fig:loss} is approximately $0.16$, consistent with the population value $1/6\approx0.1667$. This single trajectory is a numerical consistency check; it is not evidence that SGD generically reaches a global optimum.

\begin{center}
    \includegraphics[width=0.68\linewidth]{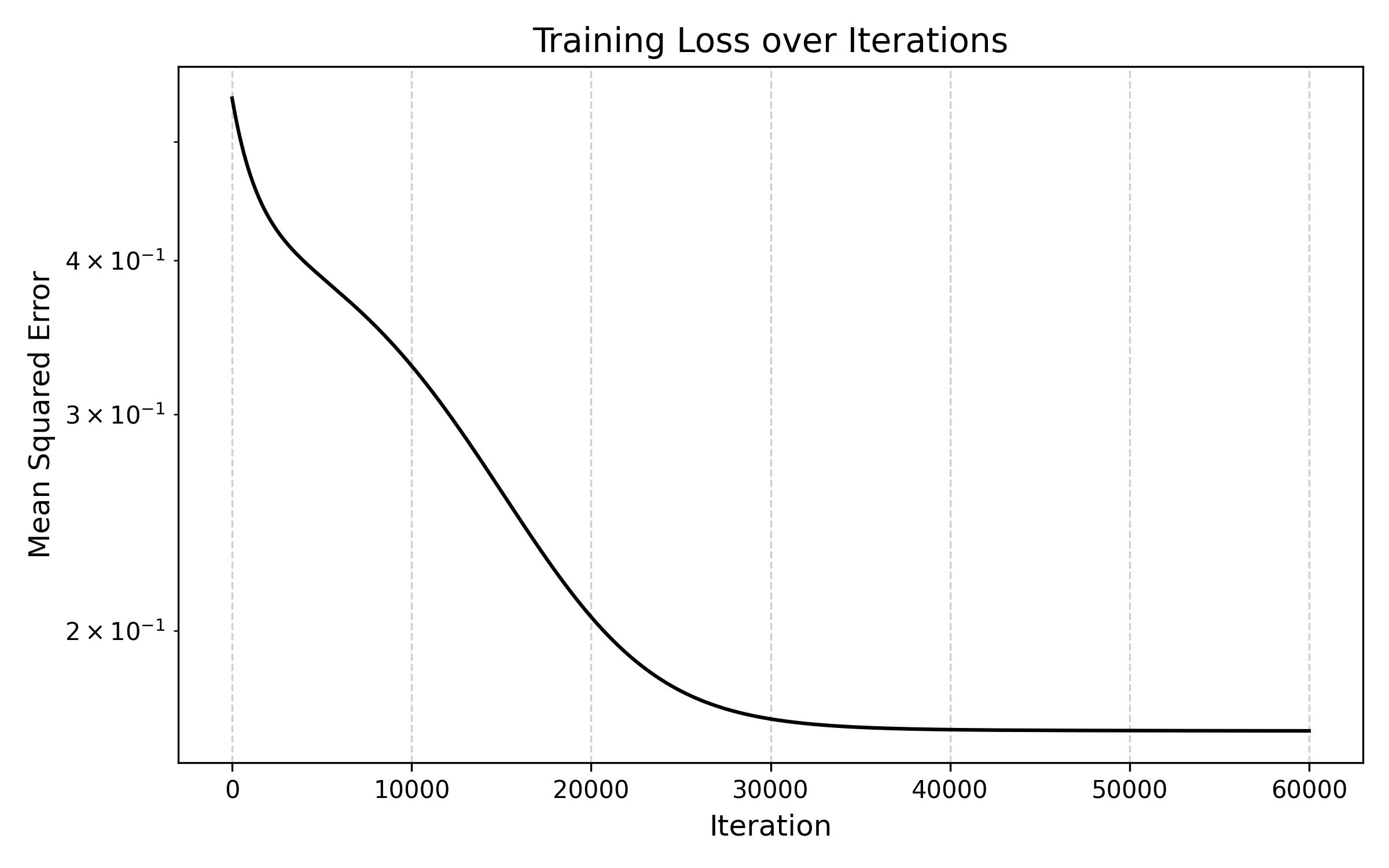}
    \captionof{figure}{Empirical mean-squared error for the recorded single-seed sphere experiment. The observed plateau is consistent with the analytic population optimum $1/6$.}
    \label{fig:loss}
\end{center}

\section{Conclusion}
\label{sec:conclusion}

The complex closure supplies equality certificates that transfer to real networks only in the nonmembership direction. Real exact rank may exceed complex rank, and the positive-semidefinite quadratic model is smaller than the signed-output architecture. A full real description may require inequalities. Likewise, border-rank membership is weaker than exact representability for $d\geq3$.

The model omits biases and depth. Adding either changes the parameterization and ambient polynomial class, so the Veronese secant description used here does not transfer without modification. The recorded benchmark is also agnostic of machine metadata; its timings are evidence about one recorded run, not portable performance claims or lower bounds for symbolic algorithms.

Taken together, the results give a simple hierarchy of certificates. The generic graph-ideal pipeline constructs equations of the border-rank locus; the target evaluator turns any violation into an exact architectural nonmembership certificate. Quadratic networks are exceptional because the bounded-rank image is closed, and their hidden orthogonal symmetry explains the parameter redundancy. In higher degree, the recovered binary-quartic and ternary-cubic catalecticant ideals agree with the known geometry. Finally, the sphere specialization shows how an algebraic obstruction combines with a distribution-induced norm to produce a quantitative loss floor.

{\footnotesize
\bibliographystyle{unsrt}
\bibliography{references}
}

\appendix
\section{Reproducibility}
\label{app:reproducibility}

This appendix records the implementation and derived data needed to reproduce the computational claims in the main text.

\paragraph{Executable pipeline.}

The general Macaulay2 construction, outlined in \cref{alg:implicitization}, is implemented in \texttt{src/fn\_power.m2}; \texttt{src/eval\_ideal.m2} exposes it for supplied $(n,r,d)$. The files \texttt{script/d2\_r2\_n2.m2} and \texttt{script/d2\_r2\_n3.m2} are fixed checks. The benchmark driver \texttt{script/expressivities.py} preserves the original grid by default and now records, when available, the Macaulay2 version, operating system, processor, logical CPU count, memory, term order, timeout, repository commit, and full architecture grid in a JSON metadata header.

\begin{algorithm}[H]
\caption{Implicitization of $\mathcal{V}_{n,d,r}$}
\label{alg:implicitization}
\begin{algorithmic}[1]
\Require integers $n,d,r\geq1$
\State $\mathcal{A}\gets\{\alpha\in\mathbb{N}^n:|\alpha|=d\}$
\State $R\gets\mathbb{Q}[\,v_{ij},c_\alpha : 1\leq i\leq r,\ 1\leq j\leq n,\ \alpha\in\mathcal{A}\,]$
\For{$\alpha\in\mathcal{A}$}
    \State $p_\alpha(v)\gets \binom{d}{\alpha}\sum_{i=1}^r v_i^\alpha$
\EndFor
\State $J\gets\langle c_\alpha-p_\alpha(v):\alpha\in\mathcal{A}\rangle\subset R$
\State $I\gets \operatorname{Elim}_{v}(J)=J\cap\mathbb{Q}[c]$
\State \Return $\operatorname{gens}(I)$
\end{algorithmic}
\end{algorithm}

The parser \path{script/analyze_benchmarks.py} reads both recorded logs and writes \path{certificate_atlas_120.csv}, \path{certificate_atlas_3600.csv}, and \path{cutoff_comparison.csv}. The exact ternary-cubic comparison is implemented in \path{script/validate_ternary_cubic.py}, with its ranks stored in \path{ternary_cubic_validation.json}. The target utility and its tests are \path{script/evaluate_certificates.py} and \path{script/test_certificate_evaluator.py}.

\paragraph{Nonzero atlas entries.}

\Cref{tab:nonzero-atlas} lists every completed architecture whose returned ideal is nonzero. The companion CSV also includes all 21 successful zero-ideal cases and all 67 timeouts. The recorded generator count is the number printed by Macaulay2; minimality was not tested.

\begin{center}
\centering
\captionof{table}{Nonzero-ideal entries in the 3600-second certificate atlas.}
\label{tab:nonzero-atlas}
\small
\begin{tabular}{@{}ccccccccc@{}}
\toprule
$(n,d,r)$ & $N$ & $rn$ & $K$ & Pred. dim. & Pred. codim. & Gens. & Deg. & Time (s) \\
\midrule
$(3,2,2)$ & 6  & 6  & 12 & 5  & 1  & 1   & 3 & 0.50 \\
$(4,2,2)$ & 10 & 8  & 18 & 7  & 3  & 10  & 3 & 0.52 \\
$(4,2,3)$ & 10 & 12 & 22 & 9  & 1  & 1   & 4 & 0.74 \\
$(5,2,2)$ & 15 & 10 & 25 & 9  & 6  & 50  & 3 & 1.03 \\
$(5,2,3)$ & 15 & 15 & 30 & 12 & 3  & 15  & 4 & 11.95 \\
$(5,2,4)$ & 15 & 20 & 35 & 14 & 1  & 1   & 5 & 202.66 \\
$(6,2,2)$ & 21 & 12 & 33 & 11 & 10 & 175 & 3 & 7.68 \\
$(6,2,3)$ & 21 & 18 & 39 & 15 & 6  & 105 & 4 & 1303.97 \\
$(3,3,2)$ & 10 & 6  & 16 & 6  & 4  & 20  & 3 & 1.21 \\
$(2,4,2)$ & 5  & 4  & 9  & 4  & 1  & 1   & 3 & 0.49 \\
$(3,4,2)$ & 15 & 6  & 21 & 6  & 9  & 148 & 3 & 38.21 \\
$(2,5,2)$ & 6  & 4  & 10 & 4  & 2  & 4   & 3 & 0.50 \\
\bottomrule
\end{tabular}
\end{center}

\paragraph{Cutoff comparison.}

The 120-second experiment contains 135 records but only 108 distinct architectures: its full 27-case $d=3$ block appears twice. The parser verifies that the duplicate records agree in status and symbolic output, collapses them by $(n,d,r)$, and records the multiplicity. Of the 60 distinct architectures shared with the one-hour grid, 58 have the same completion status. The longer log additionally completes $(5,2,4)$ in $202.66$ seconds and $(6,2,3)$ in $1303.97$ seconds. This is consistent with a cutoff effect, but the missing environment metadata prevents a controlled hardware-level timing comparison.

\paragraph{Sphere experiment.}

The supplementary file \texttt{script/sphere\_pred.py} contains the single-seed PyTorch experiment described in \cref{sec:optimization}. Samples are normalized Gaussian vectors, so they are uniform on $S^2$. The model parameter is $V\in\mathbb{R}^{2\times3}$ and the forward pass is $x\mapsto\sum_{i=1}^2(v_i^\top x)^2$. The figure reports empirical training loss, whereas Proposition~\ref{prop:sphere} concerns population loss. No multi-seed output is present in the available source code.

\end{document}